\documentclass{sn-jnl}
\usepackage{lmodern}
	\usepackage{graphicx}%
	\usepackage{multirow}%
	\usepackage{amsmath,amssymb,amsfonts}%
	\usepackage{amsthm}%
	\usepackage{mathrsfs}%
	\usepackage[title]{appendix}%
	\usepackage{xcolor}%
	\usepackage{textcomp}%
	\usepackage{manyfoot}%
	\usepackage{booktabs}%
	\usepackage{algorithm}%
	\usepackage{algorithmicx}%
	\usepackage{algpseudocode}%
	\usepackage{listings}%

	\newtheorem{theorem}{Theorem}
	\newtheorem{proposition}[theorem]{Proposition}%
	\newtheorem{remark}{Remark}%
	\newtheorem{lemma}{Lemma}%
	\newtheorem{definition}{Definition}%
	\newtheorem{corollary}{Corollary}%
\begin{document}
		
		\title[Article Title]{Approximate Controllability of Linear Fractional  Impulsive Evolution  Equations in Hilbert Spaces}
		\author[1]{\fnm{Javad} \sur{A. Asadzade}}\email{javad.asadzade@emu.edu.tr}
		
		\author[1,2]{\fnm{Nazim} \sur{I. Mahmudov}}\email{nazim.mahmudov@emu.edu.tr}
		\equalcont{These authors contributed equally to this work.}
		
		\affil[1]{\orgdiv{Department of Mathematics}, \orgname{Eastern Mediterranean University}, \orgaddress{\street{} \city{Mersin 10, 99628, T.R.}, \postcode{5380},\country{North Cyprus, Turkey}}}
		
		\affil[2]{\orgdiv{	Research Center of Econophysics}, \orgname{Azerbaijan State University of Economics (UNEC)}, \orgaddress{\street{Istiqlaliyyat Str. 6}, \city{Baku }, \postcode{1001},  \country{Azerbaijan}}}

	\abstract{
This paper investigates the approximate controllability of linear fractional impulsive evolution equations in Hilbert spaces. The system under consideration involves the Caputo fractional derivative of order $0<\alpha\leq 1$, a closed linear operator generating a strongly continuous semigroup, and instantaneous state jumps governed by bounded linear impulse operators. We first derive an explicit representation of the mild solution by combining fractional solution operators with impulsive operators. Using this representation, we characterize the approximate controllability of the system through a necessary and sufficient condition expressed in terms of the convergence of an associated family of impulsive resolvent operators. This resolvent condition extends the classical criterion for approximate controllability to the fractional impulsive setting. To illustrate the applicability of our theoretical results, a concrete example is provided. The analysis presented here bridges the gap between the well-established theory for integer-order impulsive systems and the more complex fractional case, highlighting the distinct challenges and solutions arising from the interplay of fractional dynamics and impulsive effects.}

		
		\maketitle
		
	\section{Introduction}\label{sec1}

Many natural and engineered evolutionary processes exhibit impulsive behavior, characterized by abrupt, short-term perturbations occurring at specific instants in time. Such impulsive dynamical systems are ubiquitous across disciplines including artificial intelligence, genetics, biological systems, population dynamics, neural networks, robotics, telecommunications, and computer science. For a comprehensive treatment of impulsive differential equations, we refer to the foundational monograph by Lakshmikantham and Simeonov~\cite{14}.

Within the theory of dynamical control systems, controllability is a fundamental structural property. It describes the ability to steer the system from any given initial state to any desired target state within the state space, in finite time, using an admissible set of control inputs. In recent decades, significant attention has been devoted to studying the controllability of systems whose state undergoes sudden jumps at discrete moments-so-called impulsive control systems. The literature on this topic is substantial, with key contributions addressing various aspects of controllability and observability for linear and nonlinear impulsive systems~\cite{9,10,11,12,13,18,26,27,28,29}.

This paper focuses on establishing necessary and sufficient conditions for the approximate controllability of dynamical systems governed by linear impulsive differential equations in Hilbert spaces. Our core assumption is that the principal operator \(A\), acting on the state, is the infinitesimal generator of a strongly continuous semigroup. Approximate controllability-the capability to steer the system from any initial state to an arbitrarily small neighborhood of any final state-becomes a natural and often more attainable objective in infinite-dimensional settings, due to the prevalence of non-closed subspaces~\cite{25}.

A powerful methodological approach, employed successfully in~\cite{3,8,17} successfully, reformulates the approximate controllability problem as the limit of a sequence of well-posed optimal control problems. By recasting the optimality conditions in terms of the convergence of associated resolvent operators, one obtains a concise and verifiable criterion for approximate controllability. This ``resolvent condition'' has proven to be remarkably useful and has been widely applied in the study of semilinear and impulsive evolution equations~\cite{17,21,23}. However, the introduction of impulsive effects introduces significant complications, altering the structure of the solution and the associated controllability Gramians.

The analysis becomes considerably more intricate when fractional dynamics are incorporated. Fractional differential equations, involving derivatives of non-integer order, provide a superior framework for modeling processes with inherent memory and hereditary properties, such as viscoelasticity, anomalous diffusion, and complex biological phenomena~\cite{1,2}. Controllability analysis for fractional systems must contend with the singular kernel of the fractional derivative and the fact that the solution operator forms a resolvent family rather than a semigroup~\cite{4,31}. This precludes the direct application of standard semigroup tools and necessitates the use of fractional calculus and operator theory.

Parallel to the development of impulsive systems, significant advances have been made in the controllability analysis of fractional evolution equations. Important contributions include the work of Chen, Zhang, and Li \cite{Chen,Chen2}, who investigated existence and approximate controllability for fractional systems with nonlocal conditions using resolvent operators, and Liu and Li \cite{Liu}, who established approximate controllability criteria for fractional evolution systems with Riemann–Liouville derivatives. However, these studies do not address the combined effects of fractional dynamics with impulsive discontinuities, which presents a distinct set of analytical challenges not encountered in either purely fractional or purely impulsive settings.

In a seminal work, Mahmudov~\cite{3} provided a complete characterization of the approximate controllability for linear impulsive evolution equations of the form:
\begin{equation}\label{eq:classical_impulsive}
	\begin{cases}
		x^{\prime}(t)=Ax(t)+Bu(t), & t\in[0,b]\setminus\{t_{1},\dots,t_{n}\},\\
		\Delta x(t_{k+1})=D_{k+1}x(t_{k+1})+E_{k+1}v_{k+1},& k=0,\dots,n-1,\\
		x(0)=x_{0}.
	\end{cases}
\end{equation}
The representation of the solution was given in terms of the semigroup and impulsive operators, leading to a necessary and sufficient condition for approximate controllability via an impulsive resolvent operator. These results were subsequently extended to semilinear systems via fixed-point methods~\cite{Javad1}, to finite-approximate controllability~\cite{Javad2}, and to stochastic impulsive systems with optimal control~\cite{Javad3}. Related studies on impulsive neutral and integro-differential systems further illustrate the active development in this area~\cite{Gupta,Gupta2}.

A natural and compelling question arises: \textbf{Can analogous controllability results be established for linear \emph{fractional} impulsive evolution equations?} That is, if we replace the classical derivative in~\eqref{eq:classical_impulsive} with a fractional Caputo derivative, do the core structural results on approximate controllability persist?

Addressing this question is non-trivial and reveals several profound mathematical challenges specific to the fractional setting:

\begin{enumerate}
    \item \textbf{Non-Semigroup Solution Structure:} The mild solution of a fractional evolution equation is expressed via fractional resolvent families (often involving Mittag-Leffler type functions) which lack the semigroup property. This necessitates a fundamentally different analytical framework from that used in~\cite{3}.
    
    \item \textbf{Singularity and Memory Effects:} The Caputo derivative \(^C D^\alpha_t\) has a singular kernel \((t-s)^{\alpha-1}\), embedding long-range memory into the system dynamics. This singularity complicates the analysis of solution regularity, the construction of appropriate solution operators, and the estimation of associated integrals.
    
    \item \textbf{Integration of Impulsive Jumps:} Coupling the non-local history dependence of fractional derivatives with instantaneous state jumps at impulse points creates a complex hybrid dynamic. The interplay between the ``memory'' of the fractional operator and the ``resets'' from impulses requires careful handling to define the solution properly across intervals.
\end{enumerate}

In this paper, we tackle these challenges by studying the approximate controllability of the following linear fractional impulsive evolution equation in a Hilbert space \(H\):
\begin{equation}\label{eq:fractional_main}
	\begin{cases}
		{^{C}}D^{\alpha}_{t_k^+} x(t)=Ax(t)+Bu(t), & t \in [0,\,b]\setminus\{t_1,\dots,t_n\},\\
		\Delta x(t_{k})=D_{k}x(t_{k})+E_{k}v_{k},& k=1,\dots,n,\\
		x(0)=x_{0}.
	\end{cases}
\end{equation}
Here, \(0 < \alpha \leq 1\), the state \(x(\cdot)\) takes values in \(H\), and the controls consist of a distributed control \(u \in L^2([0, b], U)\) and discrete impulsive controls \(v_k \in U\). The operator \(A: D(A) \subset H \to H\) generates a strongly continuous semigroup, and \(B \in \mathcal{L}(U,H)\), \(D_k, E_k \in \mathcal{L}(H)\) are bounded linear operators. The discontinuity at \(t_k\) is given by \(\Delta x(t_k) = x(t_k^+) - x(t_k^-)\), with the convention \(x(t_k^-) = x(t_k)\).

Our primary objectives are: (i) to derive a suitable mild solution formula for system~\eqref{eq:fractional_main} using fractional calculus and impulsive conditions; (ii) to define the corresponding controllability Gramian operator in the fractional impulsive context; and (iii) to establish a necessary and sufficient condition for approximate controllability, phrased in terms of the convergence of a family of resolvent operators. This result will serve as a rigorous fractional counterpart to the classical theory developed in~\cite{3}, highlighting both the parallels and the essential distinctions introduced by the fractional order derivative.

The remainder of this paper is structured as follows: Section~2 reviews essential preliminaries on fractional calculus, impulsive systems, and controllability concepts. Section~3 is dedicated to formulating the mild solution of system~\eqref{eq:fractional_main}. In Section~4, we present and prove our main result on the characterization of approximate controllability. Section~5 provides concluding remarks and suggests directions for future research.
	
\section{Mathematical Preliminaries}\label{sec2}

This section provides the necessary mathematical background and key definitions used throughout the paper. We introduce concepts from fractional calculus, operator theory, and impulsive systems, which are fundamental to our subsequent analysis.

\begin{definition}\cite{1}
Let \( x : [c, b] \to \mathbb{R} \) be an absolutely continuous function, and let \( \alpha \in (0, 1) \). The \emph{Caputo fractional derivative} of order \( \alpha \) with lower limit \( c \) is defined by
\[
{}^{C}D^{\alpha}_{c^+} x(t) := \frac{1}{\Gamma(1 - \alpha)} \int_{c}^{t} \frac{x'(s)}{(t - s)^{\alpha}}\, ds, \quad t > c,
\]
where \( \Gamma(\cdot) \) denotes the \emph{Gamma function}:
\[
\Gamma(\alpha) := \int_0^{\infty} t^{\alpha - 1} e^{-t} \, dt, \quad \alpha > 0.
\]
\end{definition}

Throughout this paper, \( H \) and \( U \) denote real Hilbert spaces. Let \( \mathcal{L}(H) \) be the Banach space of all bounded linear operators from \( H \) into itself, equipped with the operator norm.

\begin{definition}\cite{2}
A one-parameter family \( \{S(t)\}_{t \geq 0} \subset \mathcal{L}(H) \) is called a \emph{strongly continuous semigroup} (or \( C_0 \)-semigroup) on \( H \) if
\begin{itemize}
    \item[(i)] \( S(t)S(s) = S(t + s) \) for all \( t, s \geq 0 \);
    \item[(ii)] \( S(0) = I \), where \( I \) is the identity operator on \( H \);
    \item[(iii)] \( \lim_{t \to 0^+} S(t)x = x \) for every \( x \in H \).
\end{itemize}
The infinitesimal generator \( A : D(A) \subset H \to H \) of the semigroup is defined by
\[
Ax := \lim_{t \to 0^+} \frac{S(t)x - x}{t}, \quad x \in D(A),
\]
where \( D(A) \) consists of all \( x \in H \) for which this limit exists.
\end{definition}

The Wright function and Mittag-Leffler functions play a central role in representing solutions to fractional differential equations.

\begin{definition}\cite{1}
For \( 0 < \alpha < 1 \), the \emph{Wright function} \( \Psi_\alpha \) is defined by
\[
\Psi_\alpha(\theta) := \sum_{n=0}^{\infty} \frac{(-\theta)^n}{n! \, \Gamma(-\alpha n + 1 - \alpha)} = \frac{1}{\pi} \sum_{n=1}^{\infty} \frac{(-\theta)^n}{(n-1)! \, \Gamma(n\alpha)} \sin(n\pi\alpha), \quad \theta \in \mathbb{C}.
\]
\end{definition}

\begin{remark}\cite{1}
For \( \theta > 0 \), the Wright function can be expressed as
\[
\Psi_{\alpha}(\theta) = \frac{1}{\pi \alpha} \sum_{n=1}^{\infty} (-\theta)^{n-1} \frac{\Gamma(1+\alpha n)}{n!} \sin(n\pi\alpha), \quad 0 < \alpha < 1.
\]
\end{remark}

\begin{definition}\cite{1}
For \( \alpha > 0 \), \( \beta \in \mathbb{C} \), and \( z \in \mathbb{C} \), the \emph{generalized Mittag-Leffler function} is defined by the series
\[
E_{\alpha,\beta}(z) := \sum_{k=0}^{\infty} \frac{z^k}{\Gamma(\alpha k + \beta)}.
\]
Two frequently used special cases are
\[
E_{\alpha}(z) := E_{\alpha,1}(z) \quad \text{and} \quad e_{\alpha}(z) := E_{\alpha,\alpha}(z).
\]
\end{definition}

The following sector in the complex plane is essential for defining almost sectorial operators. For \( 0 < \mu < \pi \), let
\[
S_\mu := \{ z \in \mathbb{C} \setminus \{0\} : |\arg z| \leq \mu \} \cup \{0\},
\]
and denote its interior by \( S_\mu^0 := S_\mu \setminus \{0\} \).

Key properties of the Wright function are summarized below.

\begin{proposition}\cite{1}\label{prop:Wright_properties}
The Wright function \( \Psi_\alpha \) satisfies the following:
\begin{enumerate}
    \item[(i)] \( \Psi_\alpha(t) \geq 0 \) for \( t > 0 \).\\
    \item[(ii)] \( \displaystyle \int_0^\infty \frac{\alpha}{t^{\alpha+1}} \Psi_\alpha(t^{-\alpha}) e^{-\lambda t} \, dt = e^{-\lambda^\alpha}, \quad \operatorname{Re}(\lambda) \geq 0 \).\\
    \item[(iii)] \( \displaystyle \int_0^\infty \Psi_\alpha(t) t^r \, dt = \frac{\Gamma(1+r)}{\Gamma(1+\alpha r)}, \quad r > -1 \).\\
    \item[(iv)] \( \displaystyle \int_0^\infty \Psi_\alpha(t) e^{-z t} \, dt = E_\alpha(-z), \quad z \in \mathbb{C} \).\\
    \item[(v)] \( \displaystyle \int_0^\infty \alpha t \Psi_\alpha(t) e^{-z t} \, dt = e_\alpha(-z), \quad z \in \mathbb{C} \).
\end{enumerate}
\end{proposition}
We now introduce the solution operators associated with fractional evolution equations. Let \( A \) be the generator of a \( C_0 \)-semigroup \( \{S(t)\}_{t \geq 0} \) on \( H \). For \( t \geq 0 \), define the families of operators \( \{\mathcal{S}_\alpha(t)\}_{t \geq 0} \) and \( \{\mathcal{P}_\alpha(t)\}_{t \geq 0} \) by
\begin{align}
    \mathcal{S}_\alpha(t) &:= \int_0^\infty \Psi_\alpha(\theta) S(t^\alpha \theta) \, d\theta, \label{eq:S_alpha_def} \\
    \mathcal{P}_\alpha(t) &:= \int_0^\infty \alpha \theta \Psi_\alpha(\theta) S(t^\alpha \theta) \, d\theta. \label{eq:P_alpha_def}
\end{align}
These operators are crucial for constructing mild solutions to fractional differential equations.

In the case where \( A \) is an \emph{almost sectorial operator}, a more general class than infinitesimal generators of \( C_0 \)-semigroups, the corresponding semigroup \( \{\mathcal{Q}(t)\}_{t \geq 0} \) exhibits singular behavior at \( t = 0 \). For \( -1 < p < 0 \) and \( 0 < \omega < \frac{\pi}{2} \), we define:

\begin{definition}\cite{4}
Let \( -1 < p < 0 \) and \( 0 < \omega < \frac{\pi}{2} \). Denote by \( \Theta_p^\omega(H) \) the family of all closed linear operators \( A : D(A) \subset H \to H \) satisfying:
\begin{enumerate}
    \item[(i)] \( \sigma(A) \subset S_\omega \), where \( \sigma(A) \) is the spectrum of \( A \);
    \item[(ii)] For every \( \omega < \mu < \pi \), there exists a constant \( C_\mu > 0 \) such that
    \[
    \|R(z; A)\| \leq C_\mu |z|^p \quad \text{for all } z \in \mathbb{C} \setminus S_\mu,
    \]
    where \( R(z; A) = (zI - A)^{-1} \) is the resolvent operator of \( A \).
\end{enumerate}
An operator \( A \in \Theta_p^\omega(H) \) is called an \emph{almost sectorial operator} on \( H \).
\end{definition}

If \( A \in \Theta_p^\omega(H) \), it generates an analytic semigroup \( \{\mathcal{Q}(t)\}_{t > 0} \) of growth order \( 1+p \), which can be represented via the Laplace inversion formula:
\[
\mathcal{Q}(t) = \frac{1}{2\pi i} \int_{\Gamma_\theta} e^{-tz} R(z; A) \, dz, \quad t \in S^0_{\frac{\pi}{2} - \omega},
\]
where \( \Gamma_\theta \) is a contour surrounding the sector \( S_\omega \).

\begin{remark}\cite{4}
If \( A \in \Theta_p^\omega(H) \), the generated semigroup \( \mathcal{Q}(t) \) is analytic in an open sector but is not strongly continuous at \( t = 0 \) for general initial data. It exhibits a singularity of order \( t^{-p-1} \) as \( t \to 0^+ \).
\end{remark}

Properties of this singular semigroup and the associated fractional solution operators are given below.

\begin{proposition}\cite{4}\label{prop:almost_sectorial}
Let \( A \in \Theta_p^\omega(H) \) with \( -1 < p < 0 \) and \( 0 < \omega < \frac{\pi}{2} \). Then:
\begin{enumerate}
    \item[(i)] \( \mathcal{Q}(t) \) is analytic in \( S^0_{\frac{\pi}{2} - \omega} \), and
    \[ \frac{d^n}{dt^n} \mathcal{Q}(t) = (-A)^n \mathcal{Q}(t) \quad \text{for } t \in S^0_{\frac{\pi}{2} - \omega}, \ n \in \mathbb{N}. \]
    \item[(ii)] The semigroup property \( \mathcal{Q}(s+t) = \mathcal{Q}(s)\mathcal{Q}(t) \) holds for all \( s, t \in S^0_{\frac{\pi}{2} - \omega} \).
    \item[(iii)] There exists a constant \( c_0 = c_0(p) > 0 \) such that
    \[ \|\mathcal{Q}(t)\|_{\mathcal{L}(H)} \leq c_0 t^{-p-1} \quad \text{for all } t > 0. \]
    \item[(iv)] If \( \beta > 1 + p \), then \( D(A^\beta) \subset \Sigma_{\mathcal{Q}} := \{ x \in H : \lim_{t \to 0^+} \mathcal{Q}(t)x = x \} \).
\end{enumerate}
\end{proposition}

The corresponding fractional solution operators \( \mathcal{S}_\alpha(t) \) and \( \mathcal{P}_\alpha(t) \) (defined similarly as in \eqref{eq:S_alpha_def} and \eqref{eq:P_alpha_def} with \( S(t) \) replaced by \( \mathcal{Q}(t) \)) inherit certain boundedness and regularity properties.

\begin{proposition}\cite{4}\label{prop:fractional_op_bounds}
For each fixed \( t > 0 \), \( \mathcal{S}_\alpha(t) \) and \( \mathcal{P}_\alpha(t) \) are linear and bounded operators on \( H \). Moreover, there exist constants \( C_1, C_2 > 0 \) such that for all \( t > 0 \),
\[
\|\mathcal{S}_\alpha(t)\|_{\mathcal{L}(H)} \leq C_1 t^{-\alpha(1+p)}, \quad \|\mathcal{P}_\alpha(t)\|_{\mathcal{L}(H)} \leq C_2 t^{-\alpha(1+p)},
\]
where \( C_1 = c_0 \frac{\Gamma(-p)}{\Gamma(1-\alpha(1+p))} \) and \( C_2 = c_0 \alpha \frac{\Gamma(1-p)}{\Gamma(1-\alpha p)} \), with \( c_0 \) as in Proposition \ref{prop:almost_sectorial}.
\end{proposition}

For the main body of our work, we assume that \( A \) generates a \( C_0 \)-semigroup. In this case, the fractional solution operators are well-behaved, as summarized in the following lemma.

\begin{lemma}\cite{31}\label{lem:solution_op_properties}
Assume \( A \) generates a \( C_0 \)-semigroup \( \{S(t)\}_{t \geq 0} \) on \( H \) satisfying \( \|S(t)\| \leq M e^{\omega t} \) for some \( M \geq 1 \), \( \omega \geq 0 \). Then the operators \( \mathcal{S}_\alpha(t) \) and \( \mathcal{P}_\alpha(t) \) defined in \eqref{eq:S_alpha_def} and \eqref{eq:P_alpha_def} satisfy the following:
\begin{enumerate}
    \item[(i)] For any \( t \geq 0 \), \( \mathcal{S}_\alpha(t) \) and \( \mathcal{P}_\alpha(t) \) are bounded linear operators on \( H \). Specifically,
    \[
    \|\mathcal{S}_\alpha(t)\|_{\mathcal{L}(H)} \leq M, \quad \|\mathcal{P}_\alpha(t)\|_{\mathcal{L}(H)} \leq \frac{M}{\Gamma(\alpha)} \quad \text{for all } t \geq 0.
    \]
    \item[(ii)] The mappings \( t \mapsto \mathcal{S}_\alpha(t)x \) and \( t \mapsto \mathcal{P}_\alpha(t)x \) are continuous from \( [0,\infty) \) into \( H \) for every \( x \in H \).
    \item[(iii)] The families \( \{\mathcal{S}_\alpha(t)\}_{t \geq 0} \) and \( \{\mathcal{P}_\alpha(t)\}_{t \geq 0} \) are strongly continuous.
    \item[(iv)] If the semigroup \( \{S(t)\}_{t \geq 0} \) is compact for \( t > 0 \), then \( \mathcal{S}_\alpha(t) \) and \( \mathcal{P}_\alpha(t) \) are also compact operators for \( t > 0 \), and the mappings \( t \mapsto \mathcal{S}_\alpha(t) \) and \( t \mapsto \mathcal{P}_\alpha(t) \) are continuous in the uniform operator topology for \( t > 0 \).
\end{enumerate}
\end{lemma}

These operator families \( \mathcal{S}_\alpha(t) \) and \( \mathcal{P}_\alpha(t) \) will serve as the building blocks for constructing the mild solution of the fractional impulsive control system \eqref{eq:fractional_main} in the next section.
	\section{Main Results}\label{sec:main}

In this section, we investigate the approximate controllability of the fractional impulsive linear evolution equation \eqref{eq:fractional_main} in Hilbert spaces. First, we derive a representation of its mild solution using fractional solution operators and impulse effects. Subsequently, we establish necessary and sufficient conditions for approximate controllability in terms of an impulsive resolvent operator.

\subsection{Representation of the Mild Solution}

The following lemma provides the explicit form of the mild solution to system \eqref{eq:fractional_main}, which is essential for the subsequent controllability analysis.

\begin{lemma}\label{lem:solution_representation}
The mild solution of \eqref{eq:fractional_main} is given by
\begin{align}
    \begin{cases}
        x(t)=\mathcal{S}_{\alpha}(t)x(0)+\int_{0}^{t}(t-s)^{\alpha-1}\mathcal{P}_{\alpha}(t-s)Bu(s)\,ds, \, 0\leq t\leq t_{1},\\[6pt]
        x(t)= \mathcal{S}_{\alpha}(t-t_{k}) x(t^{+}_{k})+\int_{t_{k}}^{t}(t-s)^{\alpha-1} \mathcal{P}_{\alpha}(t-s)Bu(s)\,ds, \\ t_k < t \leq t_{k+1},\, k=1,\dots,n,
    \end{cases}
\end{align}
where the post-impulse state \( x(t_k^+) \) is recursively defined by
\begin{align}\label{eq:impulse_state}
    \begin{aligned}
    x(t^{+}_{k}) &= \prod_{j=k}^{1}(\mathcal{I}+D_{j})\mathcal{S}_{\alpha}(t_{j}-t_{j-1})x_{0} \\
    &+ \sum_{i=1}^{k}\prod_{j=k}^{i+1}(\mathcal{I}+D_{j})\mathcal{S}_{\alpha}(t_{j}-t_{j-1})
        (\mathcal{I}+D_{i})\int_{t_{i-1}}^{t_{i}}(t_{i}-s)^{\alpha-1}\mathcal{P}_{\alpha}(t_{i}-s)Bu(s)\,ds \\
    &+ \sum_{i=2}^{k}\prod_{j=k}^{i}(\mathcal{I}+D_{j}) \mathcal{S}_{\alpha}(t_{j}-t_{j-1}) E_{i-1}v_{i-1}+E_{k}v_{k}.
    \end{aligned}
\end{align}
\end{lemma}

\begin{proof}
For the interval \( [0, t_1] \), the mild solution of the fractional differential equation can be expressed using the variation of parameters formula for fractional systems:
\[
x(t)=\mathcal{S}_{\alpha}(t)x(0)+\int_{0}^{t}(t-s)^{\alpha-1}\mathcal{P}_{\alpha}(t-s)Bu(s)\,ds, \quad 0\leq t\leq t_{1}.
\]
In particular, at \( t = t_1^- \), we have
\[
x(t_1^-)=\mathcal{S}_{\alpha}(t_1)x(0)+\int_{0}^{t_1}(t_1-s)^{\alpha-1}\mathcal{P}_{\alpha}(t_1-s)Bu(s)\,ds.
\]

Using the impulsive condition \( \Delta x(t_1) = D_1 x(t_1) + E_1 v_1 \), we obtain
\begin{align*}
x(t_1^+) &= x(t_1^-) + D_1 x(t_1) + E_1 v_1 \\
&= (\mathcal{I} + D_1)\mathcal{S}_{\alpha}(t_1)x(0) 
   + (\mathcal{I} + D_1)\int_{0}^{t_1}(t_1-s)^{\alpha-1}\mathcal{P}_{\alpha}(t_1-s)Bu(s)\,ds 
   + E_1 v_1.
\end{align*}
This confirms the formula \eqref{eq:impulse_state} for \( k = 1 \).

Now, for \( t_1 < t \leq t_2 \), the solution is given by
\begin{align*}
x(t) &= \mathcal{S}_{\alpha}(t-t_1)x(t_1^+) + \int_{t_1}^{t}(t-s)^{\alpha-1}\mathcal{P}_{\alpha}(t-s)Bu(s)\,ds \\
&= \mathcal{S}_{\alpha}(t-t_1)(\mathcal{I} + D_1)\mathcal{S}_{\alpha}(t_1)x(0) \\
&+ \mathcal{S}_{\alpha}(t-t_1)(\mathcal{I} + D_1)\int_{0}^{t_1}(t_1-s)^{\alpha-1}\mathcal{P}_{\alpha}(t_1-s)Bu(s)\,ds \\
&+ \mathcal{S}_{\alpha}(t-t_1)E_1 v_1 + \int_{t_1}^{t}(t-s)^{\alpha-1}\mathcal{P}_{\alpha}(t-s)Bu(s)\,ds,
\end{align*}
which aligns with Lemma \ref{lem:solution_representation} for \( k = 1 \).

Assume that the formulas hold for some \( k = m \), i.e., for \( t_m < t \leq t_{m+1} \),
\begin{equation}\label{eq:induction_hypothesis}
x(t) = \mathcal{S}_{\alpha}(t-t_m) x(t_m^+) + \int_{t_m}^{t} (t-s)^{\alpha-1} \mathcal{P}_{\alpha}(t-s)Bu(s)\,ds,
\end{equation}
with \( x(t_m^+) \) given by \eqref{eq:impulse_state} with \( k = m \). At \( t = t_{m+1}^- \), we have
\[
x(t_{m+1}) = \mathcal{S}_{\alpha}(t_{m+1}-t_m) x(t_m^+) 
           + \int_{t_m}^{t_{m+1}} (t_{m+1}-s)^{\alpha-1} \mathcal{P}_{\alpha}(t_{m+1}-s)Bu(s)\,ds.
\]

Applying the impulsive condition at \( t_{m+1} \) yields
\begin{align*}
x(t_{m+1}^+) &= (\mathcal{I} + D_{m+1})x(t_{m+1}) + E_{m+1}v_{m+1} \\
&= (\mathcal{I} + D_{m+1})\mathcal{S}_{\alpha}(t_{m+1}-t_m) x(t_m^+) \\
& + (\mathcal{I} + D_{m+1})\int_{t_m}^{t_{m+1}} (t_{m+1}-s)^{\alpha-1} \mathcal{P}_{\alpha}(t_{m+1}-s)Bu(s)\,ds 
   + E_{m+1}v_{m+1}.
\end{align*}

Substituting the expression for \( x(t_m^+) \) from the induction hypothesis and simplifying the product terms leads to
\begin{align*}
x(t_{m+1}^+) &= \prod_{j=m+1}^{1}(\mathcal{I}+D_{j})\mathcal{S}_{\alpha}(t_{j}-t_{j-1})x_{0} \\
&+ \sum_{i=1}^{m+1}\prod_{j=m+1}^{i+1}(\mathcal{I}+D_{j})\mathcal{S}_{\alpha}(t_{j}-t_{j-1})
   (\mathcal{I}+D_{i})\int_{t_{i-1}}^{t_{i}}(t_{i}-s)^{\alpha-1}\mathcal{P}_{\alpha}(t_{i}-s)Bu(s)\,ds \\
&+ \sum_{i=2}^{m+1}\prod_{j=m+1}^{i}(\mathcal{I}+D_{j}) \mathcal{S}_{\alpha}(t_{j}-t_{j-1}) E_{i-1}v_{i-1} + E_{m+1}v_{m+1}.
\end{align*}

Consequently, for \( t_{m+1} < t \leq t_{m+2} \),
\[
x(t) = \mathcal{S}_{\alpha}(t-t_{m+1}) x(t_{m+1}^+) + \int_{t_{m+1}}^{t} (t-s)^{\alpha-1} \mathcal{P}_{\alpha}(t-s)Bu(s)\,ds,
\]
which completes the induction step. Thus, Lemma \ref{lem:solution_representation} holds for all \( k = 1,2,\dots,n \).
\end{proof}

Based on Lemma \ref{lem:solution_representation}, we can state a unified expression for the mild solution.

\begin{definition}
For \( t \in (t_k, t_{k+1}] \), \( k = 0,1,\dots,n \) (with \( t_0 = 0 \), \( t_{n+1} = b \)), the mild solution of system \eqref{eq:fractional_main} is given by
\begin{align*}
x(t) =
\begin{cases}
\displaystyle
\mathcal{S}_{\alpha}(t)x_0 + \int_0^t (t-s)^{\alpha-1} \mathcal{P}_{\alpha}(t-s)Bu(s)\,ds, \quad k = 0, \\[10pt]
\displaystyle
\mathcal{S}_{\alpha}(t-t_k) \prod_{j=k}^{1}(\mathcal{I}+D_j)\mathcal{S}_{\alpha}(t_j-t_{j-1})x_0 \\
\displaystyle + \mathcal{S}_{\alpha}(t-t_k) \sum_{i=1}^{k} \prod_{j=k}^{i+1}(\mathcal{I}+D_j)\mathcal{S}_{\alpha}(t_j-t_{j-1}) (\mathcal{I}+D_i) \\
\times\displaystyle\int_{t_{i-1}}^{t_i} (t_i-s)^{\alpha-1} \mathcal{P}_{\alpha}(t_i-s)Bu(s)\,ds \\[6pt]
\displaystyle  + \mathcal{S}_{\alpha}(t-t_k) \sum_{i=2}^{k} \prod_{j=k}^{i}(\mathcal{I}+D_j) \mathcal{S}_{\alpha}(t_j-t_{j-1}) E_{i-1}v_{i-1} + \mathcal{S}_{\alpha}(t-t_k)E_k v_k \\[6pt]
\displaystyle  + \int_{t_k}^{t} (t-s)^{\alpha-1} \mathcal{P}_{\alpha}(t-s)Bu(s)\,ds, \quad k = 1,2,\dots,n.
\end{cases}
\end{align*}
\end{definition}
Under an additional commutativity assumption between the impulse operators $ D_k$ and the solution operator $\mathcal{S}_\alpha$, the solution formula simplifies significantly.
\begin{lemma}\label{lem:commutative_case}
Assume that \( \mathcal{S}_{\alpha}(t-t_k) D_k = D_k \mathcal{S}_{\alpha}(t-t_k) \) for all \( t \in (t_k, t_{k+1}] \) and \( k = 1,2,\dots,n \). Then the mild solution of \eqref{eq:fractional_main} can be written as
\begin{align*}
x(t) =
\begin{cases}
\displaystyle
\mathcal{S}_{\alpha}(t)x_0 + \int_0^t (t-s)^{\alpha-1} \mathcal{P}_{\alpha}(t-s)Bu(s)\,ds, \quad k = 0, \\[10pt]
\displaystyle
\prod_{j=k}^{1}(\mathcal{I}+D_j)\mathcal{S}_{\alpha}(t)x_0 + \displaystyle\sum_{i=1}^{k} \prod_{j=k}^{i}(\mathcal{I}+D_j) \int_{t_{i-1}}^{t_i} (t-s)^{\alpha-1} \mathcal{P}_{\alpha}(t-s)Bu(s)\,ds \\[6pt]
\displaystyle  + \sum_{i=2}^{k} \prod_{j=k}^{i}(\mathcal{I}+D_j) \mathcal{S}_{\alpha}(t-t_{i-1}) E_{i-1}v_{i-1} + \mathcal{S}_{\alpha}(t-t_k)E_k v_k \\[6pt]
\displaystyle  + \int_{t_k}^{t} (t-s)^{\alpha-1} \mathcal{P}_{\alpha}(t-s)Bu(s)\,ds, \quad k = 1,2,\dots,n.
\end{cases}
\end{align*}
\end{lemma}

\begin{remark}
The commutativity condition \( \mathcal{S}_{\alpha}(t-t_k) D_k = D_k \mathcal{S}_{\alpha}(t-t_k) \) holds, for instance, when each \( D_k \) is a scalar multiple of the identity operator, i.e., \( D_k = d_k \mathcal{I} \) for some constant \( d_k \in \mathbb{R} \). This case frequently arises in applications where impulses correspond to uniform scaling of the state.
\end{remark}

\subsection{Adjoint Equation}

To study controllability, it is useful to consider the adjoint system associated with the homogeneous impulsive fractional equation (i.e., \eqref{eq:fractional_main} with \( B = 0 \) and \( E_k = 0 \)). The adjoint equation plays a crucial role in formulating the Gramian controllability and the associated resolvent condition.

Consider the homogeneous impulsive fractional system:
\begin{equation}\label{eq:homogeneous_impulsive}
    \begin{cases}
        {^{C}}D^{\alpha}_{t_k^+} x(t) = Ax(t), & t \in [0,b] \setminus \{t_1, \dots, t_n\}, \\
        \Delta x(t_{k}) = D_{k} x(t_{k}), & k = 1, \dots, n, \\
        x(0) = x_0.
    \end{cases}
\end{equation}

To derive the corresponding adjoint equation, we perform a formal integration by parts for the Caputo derivative. For a test function \( p(\cdot) \), we compute
\begin{align*}
    \int_{t_k}^{t_{k+1}} \langle {^{C}}D^{\alpha}_{t_k^+} x(t), p(t) \rangle \, dt 
    &= \int_{t_k}^{t_{k+1}} \langle {^{RL}}I^{1-\alpha}_{t_k^+} x'(t), p(t) \rangle \, dt \\
    &= \int_{t_k}^{t_{k+1}} \langle x'(t), {^{RL}}I^{1-\alpha}_{t_{k+1}^-} p(t) \rangle \, dt \\
    &= \langle x(t_{k+1}^-), {^{RL}}I^{1-\alpha}_{t_{k+1}^-} p(t_{k+1}^-) \rangle 
       - \langle x(t_k^+), {^{RL}}I^{1-\alpha}_{t_{k+1}^-} p(t_k^+) \rangle \\
    & + \int_{t_k}^{t_{k+1}} \langle x(t), {^{RL}}D^{\alpha}_{t_{k+1}^-} p(t) \rangle \, dt.
\end{align*}

Using equation \( {^{C}}D^{\alpha}_{t_k^+} x(t) = Ax(t) \), the left-hand side also equals
\[
 \int_{t_k}^{t_{k+1}} \langle {^{C}}D^{\alpha}_{t_k^+} x(t), p(t) \rangle \, dt =\int_{t_k}^{t_{k+1}} \langle x(t), A^* p(t) \rangle \, dt.
\]

Equating the two expressions and considering the impulsive conditions leads to the adjoint system:

\begin{equation}\label{eq:adjoint_system}
    \begin{cases}
        {^{RL}}D^{\alpha}_{t_{k+1}^-} p(t) = A^* p(t), \quad t \in (t_k, t_{k+1}), \qquad k=0,\dots,n,\\
        \Delta \left[ {^{RL}}I^{1-\alpha}_{t_{n-k+1}}p(t_{n-k+1})\right]=-D_{n-k+1}^* {^{RL}}I^{1-\alpha}_{t^+_{n-k+1}}  p(t_{n-k+1}^+), \qquad k =\overline{1,n}, \\
         {^{RL}}I^{1-\alpha}_{b^-} p(t)\Big\vert_{t=b} = \varphi.
    \end{cases}
\end{equation}

	\begin{lemma}\label{l8}
		The mild solution of the adjoint equation \eqref{eq:adjoint_system} is given by
		\begin{align}\label{45}
			\begin{cases}
p(t) = (b-t)^{\alpha-1}\mathcal{P}_{\alpha}^{*}(b-t)\varphi, \qquad
t_n < t \leq b, \\
p(t) = (t_k - t)^{\alpha-1}\mathcal{P}_{\alpha}^{*}(t_k - t)(\mathcal{I}+D_k^{*}) \\
\quad \times \prod_{i=k+1}^{n} (t_i - t_{i-1})^{\alpha-1}\mathcal{P}_{\alpha}^{*}(t_i - t_{i-1})(\mathcal{I}+D_i^{*})(b-t_n)^{\alpha-1}\mathcal{P}_{\alpha}^{*}(b-t_n)\varphi.
\end{cases}
		\end{align}
		
		where $t_{k-1}<t\leq t_{k}$, $k=n,\dots,1$.
	\end{lemma}
	\begin{proof}
		For $t_{n}<t\leq b$ the formula \eqref{45} is clear. By using induction, for $t_{n-1}<t\leq t_{n}$, we get
		\begin{align*}
			p(t)=&(t_n-t)^{\alpha-1}\mathcal{P}^{*}_{\alpha}(t_{n}-t){^{RL}}I^{1-\alpha}_{t^{-}_{n}}p(t^-_n)=(t_n-t)^{\alpha-1}\mathcal{P}^{*}_{\alpha}(t_{n}-t)(\mathcal{I}+D^{*}_{n}){^{RL}}I^{1-\alpha}_{t_n^{+}}p(t^{+}_{n})\\
			=&(t_n-t)^{\alpha-1}\mathcal{P}^{*}_{\alpha}(t_{n}-t)(\mathcal{I}+D^{*}_{n})(b-t_n)^{\alpha-1}\mathcal{P}^{*}_{\alpha}(b-t_{n})\varphi.
		\end{align*}
		Assume that, for $t_{k}<t\leq t_{k+1}$,
		\begin{align*}
			p(t)=&(t_{k+1}-t)^{\alpha-1}\mathcal{P}^{*}_{\alpha}(t_{k+1}-t)(\mathcal{I}+D^{*}_{k+1})\\
            \times&\prod_{i=k+2}^{n}(t_{i}-t_{i-1})^{\alpha-1}\mathcal{P}^{*}_{\alpha}(t_{i}-t_{i-1})(\mathcal{I}+D^{*}_{i})(b-t_{n})^{\alpha-1}\mathcal{P}^{*}_{\alpha}(b-t_{n})\varphi.
		\end{align*}
		Prove that for $t_{k-1}<t\leq t_{k}$,
		\begin{align*}
			p(t)=&(t_{k}-t)^{\alpha-1}\mathcal{P}^{*}_{\alpha}(t_{k}-t){^{RL}}I^{1-\alpha}_{t^{-}_{k}}p(t^{-}_{k})=(t_{k}-t)^{\alpha-1}\mathcal{P}^{*}_{\alpha}(t_{k}-t)(\mathcal{I}+D^{*}_{k}){^{RL}}I^{1-\alpha}_{t^{+}_{k}}p(t^{+}_{k})\\
			=&(t_{k}-t)^{\alpha-1}\mathcal{P}^{*}_{\alpha}(t_{k}-t)(\mathcal{I}+D^{*}_{k})\mathcal{P}^{*}_{\alpha}(t_{k+1}-t_{k})(\mathcal{I}+D^{*}_{k+1})\\
            \times &\prod_{i=k+2}^{n}(t_i-t_{i-1})^{\alpha-1}\mathcal{P}^{*}_{\alpha}(t_{i}-t_{i-1})(\mathcal{I}+D^{*}_{i})\mathcal{P}^{*}_{\alpha}(b-t_{n})\varphi\\
			=&(t_{k}-t)^{\alpha-1}\mathcal{P}^{*}_{\alpha}(t_{k}-t)(\mathcal{I}+D^{*}_{k})\\
            \times &\prod_{i=k+1}^{n}(t_i-t_{i-1})^{\alpha-1}\mathcal{P}^{*}_{\alpha}(t_{i}-t_{i-1})(\mathcal{I}+D^{*}_{i})(b-t_{n})^{\alpha-1}\mathcal{P}^{*}_{\alpha}(b-t_{n})\varphi.
		\end{align*}
	\end{proof}
	We now establish a Green-type formula that connects the solutions of the fractional impulsive system \eqref{eq:fractional_main} and its adjoint \eqref{eq:adjoint_system}. This formula plays a key role in the subsequent controllability analysis.

\begin{lemma}\label{lemma:green_formula}
For the mild solution \( x(t) \) of \eqref{eq:fractional_main} and the mild solution \( p(t) \) of the adjoint system \eqref{eq:adjoint_system}, the following identity holds:
\begin{equation}\label{eq:green_formula}
\begin{split}
&\quad\big\langle x(b^-), \varphi \big\rangle_H 
- \big\langle x_0,{^{RL}}I^{1-\alpha}_{t_1^-} p(0^+) \big\rangle_H\\ 
&= \sum_{k=0}^{n} \int_{t_k}^{t_{k+1}} \Big\langle u(s),\; 
   B^* (t_{k+1}-s)^{\alpha-1} \mathcal{P}_{\alpha}^*(t_{k+1}-s)\, \xi_{k+1} \Big\rangle_U ds \\
& + \sum_{k=1}^{n} \big\langle v_k,\; E_k^* \xi_{k-1} \big\rangle_U ,
\end{split}
\end{equation}
where
\begin{align}
\xi_{k+1} &= \prod_{j=k+1}^{n} \Big[ (I+D_j^*)\,(t_j-t_{j-1})^{\alpha-1} 
           \mathcal{P}_{\alpha}^*(t_j-t_{j-1}) \Big] 
           (b-t_n)^{\alpha-1} \mathcal{P}_{\alpha}^*(b-t_n) \varphi , \label{eq:xi_k} 
\end{align}
\end{lemma}

\begin{proof}
In every subinterval \( (t_k,t_{k+1}) \) we integrate by parts, using the adjoint relation between the Caputo derivative and the Riemann-Liouville integral of the right side:
\begin{align*}
\int_{t_k}^{t_{k+1}} \big\langle {}^{C}D^{\alpha}_{t_k^+} x(t),\, p(t) \big\rangle\, dt
&= \big\langle x(t_{k+1}^-),\,{^{RL}}I^{1-\alpha}_{t_{k+1}^-} p(t_{k+1}^-) \big\rangle
  - \big\langle x(t_k^+),\,{^{RL}}I^{1-\alpha}_{t_{k+1}^-} p(t_k^+) \big\rangle\\
 & + \int_{t_k}^{t_{k+1}} \big\langle x(t),\,{^{RL}}D^{\alpha}_{t_{k+1}^-} p(t) \big\rangle\, dt .
\end{align*}
Using \eqref{eq:fractional_main} and \eqref{eq:adjoint_system}, we have
\begin{align*}
    \int_{t_k}^{t_{k+1}} \big\langle Ax(t)+Bu(t),\, p(t) \big\rangle\, dt
&= \big\langle x(t_{k+1}^-),\,{^{RL}}I^{1-\alpha}_{t_{k+1}^-} p(t_{k+1}^-) \big\rangle
  - \big\langle x(t_k^+),\,{^{RL}}I^{1-\alpha}_{t_{k+1}^-} p(t_k^+) \big\rangle\\
 & + \int_{t_k}^{t_{k+1}} \big\langle x(t),\, A^*p(t) \big\rangle\, dt .
\end{align*}
Then, we get
\[
\big\langle x(t_{k+1}^-),\,{^{RL}}I^{1-\alpha}_{t_{k+1}^-} p(t_{k+1}^-) \big\rangle
- \big\langle x(t_k^+),\,{^{RL}}I^{1-\alpha}_{t_{k+1}^-} p(t_k^+) \big\rangle
= \int_{t_k}^{t_{k+1}} \big\langle u(t),\, B^* p(t) \big\rangle\, dt .
\]

Summation over all subintervals yields a telescoping sum:
\begin{align*}
&\big\langle x(b^-),\,{^{RL}}I^{1-\alpha}_{b^-} p(b^-) \big\rangle
- \big\langle x(0^+),\,{^{RL}}I^{1-\alpha}_{t_1^-} p(0^+) \big\rangle \\
+& \sum_{k=1}^{n} \Big[ \big\langle x(t_k^-),\,{^{RL}}I^{1-\alpha}_{t_k^-} p(t_k^-) \big\rangle
                     - \big\langle x(t_k^+),\,{^{RL}}I^{1-\alpha}_{t_k^-} p(t_k^+) \big\rangle \Big]\\
=& \sum_{k=0}^{n} \int_{t_k}^{t_{k+1}} \big\langle u(t),\, B^* p(t) \big\rangle\, dt .
\end{align*}

The impulsive jump condition \( x(t_k^+) = (I+D_k)x(t_k^-)+E_kv_k \) together with the adjoint jump condition \({^{RL}}I^{1-\alpha}_{t_k^-} p(t_k^-) = (I+D_k^*)\,{}^{RL}I^{1-\alpha}_{t_k^+} p(t_k^+) \) simplifies the terms at the impulse points:
\[
\big\langle x(t_k^-),\,{^{RL}}I^{1-\alpha}_{t_k^-} p(t_k^-) \big\rangle
- \big\langle x(t_k^+),\,{^{RL}}I^{1-\alpha}_{t_k^-} p(t_k^+) \big\rangle
= -\,\big\langle v_k,\; E_k^*\,{^{RL}}I^{1-\alpha}_{t_k^+} p(t_k^+) \big\rangle .
\]

Finally, inserting the explicit representation of the adjoint solution from Lemma \ref{l8},
\[
p(t) = (t_{k+1}-t)^{\alpha-1} \mathcal{P}_{\alpha}^*(t_{k+1}-t)\,\xi_{k+1} \qquad (t_k<t\le t_{k+1}),
\]
and using the terminal condition \({^{RL}}I^{1-\alpha}_{b^-} p(b^-)=\varphi \), we obtain precisely \eqref{eq:green_formula} with the coefficients \( \xi_{k+1} \) given by \eqref{eq:xi_k}.
\end{proof}

\subsubsection{Controllability of System \texorpdfstring{\eqref{eq:fractional_main}}{(main equation)}}
We now investigate the controllability properties of system \eqref{eq:fractional_main}. First, we introduce several definitions and preliminary lemmas that will be used in the analysis.
\begin{definition}\cite{3}
	Let \( H \) be a Hilbert space. A linear operator \( \mathcal{M}: H \to H \) is called positive if 
	\[
	\langle \mathcal{M} y, y \rangle \geq 0 \quad \text{for all } y \in H.
	\]
	It is said to be strictly positive if 
	\[
	\langle \mathcal{M} y, y \rangle > 0 \quad \text{for every non-zero } y \in H.
	\]
\end{definition}

For convenience, we use the notation 
\[
\Omega = (u(\cdot), \{v_{k}\}_{k=1}^{n}) \in L^{2}([0,b], U) \times U^{n}.
\]
The space \( L^{2}([0, b], U) \times U^{n} \) is a Hilbert space when endowed with the inner product \(\langle \cdot, \cdot \rangle_{1}\) defined by
\[
\big\langle \Omega_{1}, \Omega_{2} \big\rangle_{1} 
= \int_{0}^{b} \big\langle u_{1}(s), u_{2}(s) \big\rangle_{U} \, ds 
+ \sum_{k=1}^{n} \big\langle v_{1k}, v_{2k} \big\rangle_{U},
\]
for any \(\Omega_{1}, \Omega_{2} \in L^{2}([0,b], U) \times U^{n}\).

To investigate the controllability of an impulsive system, we introduce the bounded linear operator  
\[
\mathcal{M}: L^2([0, b], U) \times U^n \to H
\]
defined by  
\begin{align}\label{tu}
&\mathcal{M}(u(\cdot),\{v_{k}\}_{k=1}^{n})=\int_{t_{n}}^{b}(b-s)^{\alpha-1}\mathcal{P}_{\alpha}(b-s)Bu(s)ds\nonumber\\
&+\mathcal{S}_{\alpha}(b-t_{n})\sum_{i=1}^{n}\prod_{j=n}^{i+1}(\mathcal{I}+D_{j})\mathcal{S}_{\alpha}(t_{j}-t_{j-1})
(\mathcal{I}+D_{i})\int_{t_{i-1}}^{t_{i}}(t_{i}-s)^{\alpha-1}\mathcal{P}_{\alpha}(t_{i}-s)Bu(s)ds\nonumber\\
&+\mathcal{S}_{\alpha}(b-t_{n})\sum_{i=2}^{n}\prod_{j=n}^{i}(\mathcal{I}+D_{j}) \mathcal{S}_{\alpha}(t_{j}-t_{j-1}) E_{i-1}v_{i-1}+\mathcal{S}_{\alpha}(b-t_{n})E_{n}v_{n}.
\end{align}

The following lemma establishes the well-posedness of \(\mathcal{M}\) as an operator from \(L^2([0, b], U) \times U^n\) into \(H\).
\bigskip

\begin{lemma}
The operator \(\mathcal{M}: L^2([0, b], U) \times U^n \to H\) defined in \eqref{tu} is bounded and linear. Moreover, \(\mathcal{M}(u(\cdot), \{v_k\}_{k=1}^{n}) \in L^2([0, b], H)\) for every \(u(\cdot) \in L^2([0, b], U)\) and \(\{v_k\}_{k=1}^{n} \in U^n\).
\end{lemma}

\begin{proof}
\textbf{(i)\, Linearity:}
Let \((u_1(\cdot), \{v_k^1\}_{k=1}^{n})\) and \((u_2(\cdot), \{v_k^2\}_{k=1}^{n})\) be arbitrary elements of \(L^2([0, b], U) \times U^n\), and let \(\lambda \in \mathbb{R}\). By the linearity of the integral, the linearity of the operators \(\mathcal{P}_{\alpha}\), \(\mathcal{S}_{\alpha}\), \(B\), \(D_j\), and \(E_j\), and the distributive property of summation, it follows directly that  
\[
\mathcal{M}(u_1 + \lambda u_2, \{v_k^1 + \lambda v_k^2\}_{k=1}^{n}) = \mathcal{M}(u_1, \{v_k^1\}_{k=1}^{n}) + \lambda \mathcal{M}(u_2, \{v_k^2\}_{k=1}^{n}).
\]
Thus, \(\mathcal{M}\) is a linear operator.

\textbf{(ii)\, Boundedness: } 
We now establish the boundedness of \(\mathcal{M}\). Let \(M, C_B, C_D, C_E\) denote positive constants such that  
\[
\|\mathcal{P}_{\alpha}(t)\| \leq \frac{M}{\Gamma(\alpha)}, \quad \|\mathcal{S}_{\alpha}(t)\| \leq M, \quad \|B\| \leq C_B, \quad \|D_j\| \leq C_D, \quad \|E_j\| \leq C_E
\]
for all \(t \in [0, b]\) and \(j = 1, \dots, n\). Using these bounds and applying Lemma \ref{lem:solution_op_properties}, we obtain  
\begin{align*}
\|\mathcal{M}(u(\cdot),\{v_{k}\}_{k=1}^{n})\|_{H} &\leq \frac{MC_{B}}{\Gamma(\alpha)}\int_{t_{n}}^{b}(b-s)^{\alpha-1}\| u(s)\|_{U}\,ds \\
&+ \sum_{i=1}^n (1+C_D)^{n-i+1}M^{n-i+2}\frac{C_B}{\Gamma(\alpha)}\int_{t_{i-1}}^{t_i}(t_i-s)^{\alpha-1}\| u(s)\|_{U}\, ds \\
&+ \sum_{i=1}^n (1+C_D)^{n-i}M^{n-i+1}C_E \| v_i\|_{U}.
\end{align*}
Let \(t_0 = 0\) and \(t_{n+1} = b\). Then the inequality above can be rewritten as  
\[
\|\mathcal{M}(u(\cdot),\{v_{k}\}_{k=1}^{n})\|_{H} \leq C_1 \sum_{i=1}^{n+1} \int_{t_{i-1}}^{t_i}(t_i-s)^{\alpha-1}\| u(s)\|_{U}\, ds + C_2 \sum_{i=1}^{n} \| v_i\|_{U},
\]
where \(C_1, C_2 > 0\) are appropriate constants. Applying the Cauchy--Schwarz inequality to each integral term yields  
\[
\int_{t_{i-1}}^{t_i}(t_i-s)^{\alpha-1}\| u(s)\|_{U}\, ds \leq \left( \int_{t_{i-1}}^{t_i}(t_i-s)^{2\alpha-2}\, ds \right)^{\frac{1}{2}} \| u \|_{L^2([t_{i-1},t_i],U)}.
\]
Since  
\[
\int_{t_{i-1}}^{t_i}(t_i-s)^{2\alpha-2}\, ds = \frac{(t_i - t_{i-1})^{2\alpha-1}}{2\alpha-1} \leq \frac{b^{2\alpha-1}}{2\alpha-1},
\]
we obtain  
\[
\|\mathcal{M}(u(\cdot),\{v_{k}\}_{k=1}^{n})\|_{H} \leq C_1 \frac{b^{\alpha-\frac{1}{2}}}{\sqrt{2\alpha-1}} \sum_{i=1}^{n+1} \| u \|_{L^2([t_{i-1},t_i],U)} + C_2 \sum_{i=1}^{n} \| v_i\|_{U}.
\]
Finally, using the fact that  
\[
\sum_{i=1}^{n+1} \| u \|_{L^2([t_{i-1},t_i],U)} \leq(n+1) \, \| u \|_{L^2([0,b],U)},
\]
we conclude that there exists a constant \(\tilde{C} > 0\) such that  
\[
\|\mathcal{M}(u(\cdot),\{v_{k}\}_{k=1}^{n})\|_{H} \leq \tilde{C} \left( \| u \|_{L^2([0,b],U)} + \sum_{k=1}^{n} \| v_k\|_{U} \right).
\]
Hence, \(\mathcal{M}\) is a bounded linear operator. Since \(\mathcal{M}(u(\cdot), \{v_k\}_{k=1}^{n})\) belongs to \(H\) for each input pair, the boundedness ensures that the operator maps into \(L^2([0, b], H)\), completing the proof.
\end{proof}

\begin{lemma}\label{lemmaadj}
The adjoint operator \(\mathcal{M}^*\) is given by
\[
\mathcal{M}^*\varphi = \bigl( B^*p(\cdot),\; \{ E_k^*p(t_k^+)\}_{k=1}^n \bigr),
\]
where
\[
B^*p(t)=
\begin{cases}
\begin{aligned}
&B^*(b-t)^{\alpha-1}\mathcal{P}_{\alpha}^*(b-t)\varphi, & t_n<t\le b,\\
& B^*(t_k-t)^{\alpha-1}\mathcal{P}_{\alpha}^*(t_k-t)(\mathcal{I}+D_k^*) \prod_{i=k+1}^{n}(t_i-t_{i-1})^{\alpha-1} \mathcal{P}_{\alpha}^*(t_i-t_{i-1}) \\
&(\mathcal{I}+D_i^*)\,(b-t_n)^{\alpha-1}\mathcal{P}_{\alpha}^*(b-t_n)\varphi, & t_{k-1}<t\le t_k,
\end{aligned}
\end{cases}
\]
and
\[
E_k^*p(t_k^+)=
\begin{cases}
\begin{aligned}
&E_n^*(b-t_n)^{\alpha-1}\mathcal{P}_{\alpha}^*(b-t_n)\varphi, \quad k=n,\\
& E_k^*(t_k-t_{k-1})^{\alpha-1}\mathcal{P}_{\alpha}^*(t_k-t_{k-1})(\mathcal{I}+D_k^*) \prod_{i=k+1}^{n}(t_i-t_{i-1})^{\alpha-1} \mathcal{P}_{\alpha}^*(t_i-t_{i-1}) \\
&(\mathcal{I}+D_i^*)\,(b-t_n)^{\alpha-1}\mathcal{P}_{\alpha}^*(b-t_n)\varphi, \quad k=n-1,\dots,1.
\end{aligned}
\end{cases}
\]
\end{lemma}

\begin{proof}
Substituting \(x(0)=0\) into the Green-type formula \eqref{eq:green_formula} we obtain
\begin{align*}
    \langle x(b),\varphi\rangle
= \bigl\langle \Omega,\; \mathcal{M}^*p(b)\bigr\rangle_1
&= \sum_{k=0}^{n} \int_{t_k}^{t_{k+1}} \Bigl\langle u(s),\; 
B^* (t_{k+1}-s)^{\alpha-1} \mathcal{P}_{\alpha}^*(t_{k+1}-s)\, \xi_{k+1} \Bigr\rangle_U\,ds\\
&+ \sum_{k=1}^{n} \bigl\langle v_k,\; E_k^* \xi_{k-1} \bigr\rangle_U,
\end{align*}

with \(\xi_k\) defined as in \eqref{eq:xi_k}, respectively. Consequently,
\begin{align*}
&\qquad\langle x(b),\varphi\rangle
= \Bigl\langle \int_{t_n}^{b}(b-s)^{\alpha-1}\mathcal{P}_{\alpha}(b-s)Bu(s)\,ds,\;\varphi\Bigr\rangle \\
&+ \Bigl\langle \mathcal{S}_{\alpha}(b-t_n) \sum_{i=1}^{n}\prod_{j=n}^{i+1}(\mathcal{I}+D_j)\mathcal{S}_{\alpha}(t_j-t_{j-1})
      (\mathcal{I}+D_i)\int_{t_{i-1}}^{t_i}(t_i-s)^{\alpha-1}\mathcal{P}_{\alpha}(t_i-s)Bu(s)\,ds,\;\varphi\Bigr\rangle \\
&+ \Bigl\langle \mathcal{S}_{\alpha}(b-t_n)\sum_{i=2}^{n}\prod_{j=n}^{i}(\mathcal{I}+D_j)\mathcal{S}_{\alpha}(t_j-t_{j-1})E_{i-1}v_{i-1},\;\varphi\Bigr\rangle
   + \bigl\langle \mathcal{S}_{\alpha}(b-t_n)E_nv_n,\;\varphi\bigr\rangle \\
&= \int_{t_n}^{b} \bigl\langle (b-s)^{\alpha-1}\mathcal{P}_{\alpha}(b-s)Bu(s),\;\varphi\bigr\rangle\,ds \\
&+ \sum_{i=1}^{n} \int_{t_{i-1}}^{t_i} \Bigl\langle (t_i-s)^{\alpha-1}\mathcal{P}_{\alpha}(t_i-s)Bu(s),\; 
      (\mathcal{I}+D_i^*)\prod_{j=i+1}^{n}\mathcal{S}_{\alpha}^*(t_j-t_{j-1})(\mathcal{I}+D_j^*)\mathcal{S}_{\alpha}^*(b-t_n)\varphi\Bigr\rangle\,ds \\
&+ \sum_{i=2}^{n} \Bigl\langle v_{i-1},\; E_{i-1}^*\prod_{j=i}^{n}\mathcal{S}_{\alpha}^*(t_j-t_{j-1})(\mathcal{I}+D_j^*)\mathcal{S}_{\alpha}^*(b-t_n)\varphi\Bigr\rangle
   + \bigl\langle v_n,\; E_n^*\mathcal{S}_{\alpha}^*(b-t_n)\varphi\bigr\rangle \\
&= \int_{t_n}^{b} \bigl\langle u(s),\; B^*(b-s)^{\alpha-1}\mathcal{P}_{\alpha}^*(b-s)\varphi\bigr\rangle\,ds \\
&+ \sum_{i=1}^{n} \int_{t_{i-1}}^{t_i} \Bigl\langle u(s),\; B^*(t_i-s)^{\alpha-1}\mathcal{P}_{\alpha}^*(t_i-s)(\mathcal{I}+D_i^*)
      \prod_{j=i+1}^{n}\mathcal{S}_{\alpha}^*(t_j-t_{j-1})(\mathcal{I}+D_j^*)\mathcal{S}_{\alpha}^*(b-t_n)\varphi\Bigr\rangle\,ds \\
&+ \sum_{i=2}^{n} \Bigl\langle v_{i-1},\; E_{i-1}^*\prod_{j=i}^{n}\mathcal{S}_{\alpha}^*(t_j-t_{j-1})(\mathcal{I}+D_j^*)\mathcal{S}_{\alpha}^*(b-t_n)\varphi\Bigr\rangle + \bigl\langle v_n,\; E_n^*\mathcal{S}_{\alpha}^*(b-t_n)\varphi\bigr\rangle.
\end{align*}
\end{proof}
	\begin{lemma}
		The operator $\mathcal{M}\mathcal{M}^{*}$ has the following form
		\begin{align*}
\mathcal{M}\mathcal{M}^{*}=\Omega_{t_{n}}^{b}+\Psi^{t_{n}}_{0}+\tilde{\Omega}_{t_{n}}^{b}+\tilde{\Psi}^{t_{n}}_{0},
		\end{align*}
		where $\Omega_{t_{n}}^{b},\Psi^{t_{n}}_{0}, \tilde{\Omega}_{t_{n}}^{b},\tilde{\Psi}^{t_{n}}_{0}: H\to H$ are non-negative operators and define as follows
		\begin{align*}
			\Omega_{t_{n}}^{b}:=&\int_{t_{n}}^{b}(b-s)^{2\alpha-2}\mathcal{P}_{\alpha}(b-s)BB^{*}\mathcal{P}^{*}_{\alpha}(b-s)\varphi ds,
		\end{align*}
		\begin{align*}
			\tilde{\Omega}_{t_{n}}^{b}:=&\mathcal{S}_{\alpha}(b-t_{n})E_{n}E^{*}_{n}\mathcal{S}^{*}_{\alpha}(b-t_{n})\varphi,
		\end{align*}
		\begin{align*}
			\Psi^{t_{n}}_{0}&:=\mathcal{S}_{\alpha}(b-t_{n})\sum_{i=1}^{n}\prod_{j=n}^{i+1}(\mathcal{I}+D_{j})\mathcal{S}_{\alpha}(t_{j}-t_{j-1})
			(\mathcal{I}+D_{i})\\
            &\times \int_{t_{i-1}}^{t_{i}}(t_i-s)^{2\alpha-2}\mathcal{P}_{\alpha}(t_{i}-s)BB^{*}\mathcal{P}^{*}_{\alpha}(t_{i}-s)\,ds\nonumber\\
			&\times (\mathcal{I}+D^{*}_{i})\prod_{j=i+1}^{n}
			\mathcal{S}^{*}_{\alpha}(t_{j}-t_{j-1})(\mathcal{I}+D^{*}_{j})\mathcal{S}^{*}_{\alpha}(b-t_{n})\varphi,
		\end{align*}
		
		\begin{align*}
			\tilde{\Psi}^{t_{n}}_{0}:=&\mathcal{S}_{\alpha}(b-t_{n})\sum_{i=2}^{n}\prod_{j=n}^{i}(\mathcal{I}+D_{j}) \mathcal{S}_{\alpha}(t_{j}-t_{j-1}) E_{i-1}\\
            \times &E^{*}_{i-1}\prod_{j=i}^{n} \mathcal{S}^{*}_{\alpha}(t_{j}-t_{j-1}) (\mathcal{I}+D^{*}_{j})\mathcal{S}^{*}_{\alpha}(b-t_{n})\varphi.
		\end{align*}
	\end{lemma}
	\begin{proof}
		In fact
		\begin{align*}
			&\mathcal{M}\mathcal{M}^{*}\varphi=\mathcal{M}\left(B^{*}p(\cdot), \big\{E^{*}_{k}p(t^{+}_{k})\big\}_{k=1}^{n}\right)\\
			&=\int_{t_{n}}^{b}(b-s)^{2\alpha-2}\mathcal{P}_{\alpha}(b-s)BB^{*}\mathcal{P}^{*}_{\alpha}(b-s)\varphi ds\\
            &+\mathcal{S}_{\alpha}(b-t_{n})\sum_{i=1}^{n}\prod_{j=n}^{i+1}(\mathcal{I}+D_{j})\mathcal{S}_{\alpha}(t_{j}-t_{j-1})
			(\mathcal{I}+D_{i})\\
			&\times\int_{t_{i-1}}^{t_{i}}(t_{i}-s)^{2\alpha-2}\mathcal{P}_{\alpha}(t_{i}-s)BB^{*}\mathcal{P}^{*}_{\alpha}(t_{i}-s)\,ds\\
            &\times (\mathcal{I}+D^{*}_{i})\prod_{j=i+1}^{n}
			\mathcal{S}^{*}_{\alpha}(t_{j}-t_{j-1})(\mathcal{I}+D^{*}_{j})\mathcal{S}^{*}_{\alpha}(b-t_{n})\varphi\\
			&+\mathcal{S}_{\alpha}(b-t_{n})\sum_{i=2}^{n}\prod_{j=n}^{i}(\mathcal{I}+D_{j}) \mathcal{S}_{\alpha}(t_{j}-t_{j-1}) E_{i-1}\\
            &\times E^{*}_{i-1}\prod_{j=i}^{n} \mathcal{S}^{*}_{\alpha}(t_{j}-t_{j-1}) (\mathcal{I}+D^{*}_{j})\mathcal{S}^{*}_{\alpha}(b-t_{n})\varphi\nonumber\\
			&+\mathcal{S}_{\alpha}(b-t_{n})E_{n}E^{*}_{n}\mathcal{S}^{*}_{\alpha}(b-t_{n})\varphi\nonumber\\
			&=\int_{t_{n}}^{b}(b-s)^{2\alpha-2}\mathcal{P}_{\alpha}(b-s)BB^{*}\mathcal{P}^{*}_{\alpha}(b-s)\varphi \,ds\\
            &+\mathcal{S}_{\alpha}(b-t_{n})\sum_{i=1}^{n}\prod_{j=n}^{i+1}(\mathcal{I}+D_{j})\mathcal{S}_{\alpha}(t_{j}-t_{j-1})
			(\mathcal{I}+D_{i})\nonumber\\
			&\times\int_{t_{i-1}}^{t_{i}}(t_i-s)^{2\alpha-2}\mathcal{P}_{\alpha}(t_{i}-s)BB^{*}\mathcal{P}^{*}_{\alpha}(t_{i}-s)\,ds\\
            &\times (\mathcal{I}+D^{*}_{i})\prod_{j=i+1}^{n}
			\mathcal{S}^{*}_{\alpha}(t_{j}-t_{j-1})(\mathcal{I}+D^{*}_{j})\mathcal{S}^{*}_{\alpha}(b-t_{n})\varphi\\
			&+\mathcal{S}_{\alpha}(b-t_{n})\sum_{i=2}^{n}\prod_{j=n}^{i}(\mathcal{I}+D_{j}) \mathcal{S}_{\alpha}(t_{j}-t_{j-1}) E_{i-1}\\
            &\times E^{*}_{i-1}\prod_{j=i}^{n} \mathcal{S}^{*}_{\alpha}(t_{j}-t_{j-1}) (\mathcal{I}+D^{*}_{j})\mathcal{S}^{*}_{\alpha}(b-t_{n})\varphi\\
			&+\mathcal{S}_{\alpha}(b-t_{n})E_{n}E^{*}_{n}\mathcal{S}^{*}_{\alpha}(b-t_{n})\varphi.
		\end{align*}
		Clearly, $\Omega_{t_{n}}^{b},\Psi^{t_{n}}_{0}, \tilde{\Omega}_{t_{n}}^{b},\tilde{\Psi}^{t_{n}}_{0}: H\to H$ are non-negative.
	\end{proof}
	\begin{remark}
		In the non-impulsive case, the expression for \(\mathcal{M}\mathcal{M}^{*}\) simplifies to
		\begin{align*}
			\mathcal{M}\mathcal{M}^{*} = \Omega^{b}_{0} := \int_{0}^{b} (b-s)^{2\alpha-2}\mathcal{P}_{\alpha}(b-s)BB^{*}\mathcal{P}^{*}_{\alpha}(b-s)\varphi \, ds.
		\end{align*}
		In this scenario, there is only one controllability operator.
	\end{remark}
	\begin{definition}
		The system \eqref{eq:fractional_main} is said to be approximately controllable on $[0, b]$
		if $\overline{Im \mathcal{M}} =H$.
	\end{definition} 
	\begin{theorem}\label{t13}
		The following conditions are equivalent.
		\\
		(i) System \eqref{eq:fractional_main} is approximately controllable on $[0,b]$.
		\\
		(ii) $\mathcal{M}^{*}\varphi=0$ implies that $\varphi=0$.
		\\
		(iii) $\Omega_{t_{n}}^{b}+\Psi^{t_{n}}_{0}+\tilde{\Omega}_{t_{n}}^{b}+\tilde{\Psi}^{t_{n}}_{0}$ is strictly positive.
		\\
		(iv) \[
		\varepsilon \left( \varepsilon I + \Omega_{t_{n}}^{b}+\Psi^{t_{n}}_{0}+\tilde{\Omega}_{t_{n}}^{b}+\tilde{\Psi}^{t_{n}}_{0} \right)^{-1}
		\]
		converges to the zero operator as \( \varepsilon \rightarrow 0^+ \) in the strong operator topology.
		\\
		(v) \( \varepsilon \left( \varepsilon \mathcal{I} + \Omega_{t_{n}}^{b}+\Psi^{t_{n}}_{0}+\tilde{\Omega}_{t_{n}}^{b}+\tilde{\Psi}^{t_{n}}_{0} \right)^{-1} \)
		converges to the zero operator as \( \varepsilon \rightarrow 0^+ \) in the weak operator topology.
	\end{theorem}
	\begin{proof}
		The equivalence between \( (i) \Leftrightarrow (ii) \) is standard. The approximately controllability of system \eqref{eq:fractional_main} on \([0, b]\) is equivalent to the density of \( \text{Im} \mathcal{M} \) in \( H \). In other words, the kernel of \( \mathcal{M}^* \) is trivial in \( H \). Similarly, 
		\begin{align*}
			\mathcal{M}^{*}\varphi=\left(B^{*}p(\cdot), \big\{E^{*}_{k}p(t^{+}_{k})\big\}_{k=1}^{n}\right)=0
		\end{align*}
		implies \( \varphi = 0 \). The equivalence \( (i) \Leftrightarrow (iii) \) is a well-established result, (see \cite{5} page 207). The equivalence \( (iv) \Leftrightarrow (v) \) stems from the positivity of
		\[ \varepsilon \left( \varepsilon I + \Omega_{t_{n}}^{b}+\Psi^{t_{n}}_{0}+\tilde{\Omega}_{t_{n}}^{b}+\tilde{\Psi}^{t_{n}}_{0} \right)^{-1}. \]
		We prove only \( (i) \Rightarrow (iv) \). To achieve this, consider the functional
		\begin{align*}
			\mathcal{J}_{\varepsilon}(\varphi)=\frac{1}{2}\Vert \mathcal{M}^{*}\varphi\Vert^{2}+\frac{\varepsilon}{2}\Vert\varphi\Vert^{2}-\Big\langle \varphi, h- \mathcal{S}_{\alpha}(b-t_{n})\prod_{j=n}^{1}(1+D_{j})\mathcal{S}_{\alpha}(t_{j}-t_{j-1})x_{0}\Big\rangle.
		\end{align*}
		The mapping \( \varphi \rightarrow \mathcal{J}_{\varepsilon}(\varphi) \) is continuous and strictly convex. The functional \( \mathcal{J}_{\varepsilon}(\cdot) \) possesses a unique minimum \( \tilde{\varphi}_{\varepsilon} \) that defines a mapping \( \Phi : H \rightarrow H \). Since \( \mathcal{J}_{\varepsilon}(\varphi) \) is Frechet differentiable at \( \tilde{\varphi}_{\varepsilon} \), by the optimality of \( \tilde{\varphi}_{\varepsilon} \), we must have
		\begin{align}\label{11}
			\frac{d}{d\varphi}\mathcal{J}_{\varepsilon}(\varphi)=&\Omega_{t_{n}}^{b}\tilde{\varphi}_{\varepsilon}+\Psi^{t_{n}}_{0}\tilde{\varphi}_{\varepsilon}+\tilde{\Omega}_{t_{n}}^{b}\tilde{\varphi}_{\varepsilon}+\tilde{\Psi}^{t_{n}}_{0}\tilde{\varphi}_{\varepsilon}+\varepsilon\tilde{\varphi}_{\varepsilon}-h\nonumber\\
			+&\mathcal{S}_{\alpha}(b-t_{n})\prod_{j=n}^{1}(1+D_{j})\mathcal{S}_{\alpha}(t_{j}-t_{j-1})x_{0}=0.
		\end{align}
		
		Solving \eqref{11} for $\tilde{\varphi}_{\varepsilon}$, we obtain
		\begin{align}\label{rtu12}
			\tilde{\varphi}_{\varepsilon}=\left( \varepsilon I + \Omega_{t_{n}}^{b}+\Psi^{t_{n}}_{0}+\tilde{\Omega}_{t_{n}}^{b}+\tilde{\Psi}^{t_{n}}_{0} \right)^{-1}\left(h-\mathcal{S}_{\alpha}(b-t_{n})\prod_{j=n}^{1}(1+D_{j})\mathcal{S}_{\alpha}(t_{j}-t_{j-1})x_{0}\right).
		\end{align}
		Defining \( u_{\varepsilon}(s) \) and \( \{ v_{\varepsilon}^k \}_{k=1}^n \) as follows: 
		
		\begin{align}\label{control}
        \begin{cases}
            u_{\varepsilon}(s) =\bigg(\displaystyle\sum_{k=1}^{n} B^* (t_k-s)^{\alpha-1}\mathcal{P}^{*}_{\alpha} (t_{k} - s)  (\mathcal{I}+D^{*}_{k})\\
            \displaystyle\prod_{i=k+1}^{n} \mathcal{S}^{*}_{\alpha}(t_i - t_{i-1}) (\mathcal{I}+D^{*}_{i})\mathcal{S}^{*}_{\alpha}(b - t_{n}) \chi(t_{k-1}, t_k)  \\           + B^{*} (b-s)^{\alpha-1}\mathcal{P}^{*}_{\alpha} (b - s) \chi(t_{n}, b)\bigg) \tilde{\varphi}_{\varepsilon}, \\
			v^{\varepsilon}_{n} = E^{*}_{n} \mathcal{S}^{*}_{\alpha} (b - t_{n}) \tilde{\varphi}_{\varepsilon}, \\
			v_{\varepsilon}^k = E^{*}_{k}\displaystyle\prod_{i=k}^{n} \mathcal{S}^{*}_{\alpha}(t_{i}-t_{i-1}) (\mathcal{I}+D^{*}_{i})\mathcal{S}^{*}_{\alpha}(b-t_{n})\tilde{\varphi}_{\varepsilon}, \,\, k = 1, \ldots, n-1,
        \end{cases}
		\end{align}
		
		we obtain from \eqref{11} and \eqref{rtu12} that 
		\begin{align}\label{47}
			x_{\varepsilon}(b)-h=&-\varepsilon \tilde{\varphi}_{\varepsilon}=-\varepsilon\left( \varepsilon I + \Omega_{t_{n}}^{b}+\Psi^{t_{n}}_{0}+\tilde{\Omega}_{t_{n}}^{b}+\tilde{\Psi}^{t_{n}}_{0} \right)^{-1}\nonumber\\
			\times &\left(h-\mathcal{S}^{*}_{\alpha}(b-t_{n})\prod_{j=n}^{1}(1+D_{j})\mathcal{S}_{\alpha}(t_{j}-t_{j-1})x_{0}\right),
		\end{align}
		where 
		\begin{align*}
			x_{\varepsilon}(b)=x\left(b; x_{0}, u^{\varepsilon}, \big\{v^{\varepsilon}_{k}\big\}_{k=1}^{n}\right)=&\mathcal{S}_{\alpha}(b-t_{n})\prod_{j=n}^{1}(1+D_{j})\mathcal{S}_{\alpha}(t_{j}-t_{j-1})x_{0}\\
+&\Omega_{t_{n}}^{b}\tilde{\varphi}_{\varepsilon}+\Psi^{t_{n}}_{0}\tilde{\varphi}_{\varepsilon}+\tilde{\Omega}_{t_{n}}^{b}\tilde{\varphi}_{\varepsilon}+\tilde{\Psi}^{t_{n}}_{0}\tilde{\varphi}_{\varepsilon}+\varepsilon\tilde{\varphi}_{\varepsilon}.
		\end{align*}
		The equivalence of $(i)$ and $(iv)$ follows directly from \eqref{47}.
	\end{proof}

	\begin{corollary} If one of the operators $\Omega_{t_{n}}^{b},\Psi^{t_{n}}_{0}, \tilde{\Omega}_{t_{n}}^{b},\tilde{\Psi}^{t_{n}}_{0}$ are strictly positive, then the combined operator $\Omega_{t_{n}}^{b}+\Psi^{t_{n}}_{0}+ \tilde{\Omega}_{t_{n}}^{b}+\tilde{\Psi}^{t_{n}}_{0}$ is also strictly positive. Consequently, the system \eqref{eq:fractional_main} is approximately controllable on the interval $[0, b]$.
	\end{corollary}
	\begin{proof}
		The operators \(\Omega_{t_{n}}^{b}, \Psi^{t_{n}}_{0}, \tilde{\Omega}_{t_{n}}^{b}, \tilde{\Psi}^{t_{n}}_{0}\) are all non-negative. Hence, if one of these operators is strictly positive, their sum \(\Omega_{t_{n}}^{b} + \Psi^{t_{n}}_{0} + \tilde{\Omega}_{t_{n}}^{b} + \tilde{\Psi}^{t_{n}}_{0}\) will also be strictly positive. Using Theorem \(\ref{t13}\) $(iii)$, we can therefore infer that the system \(\eqref{eq:fractional_main}\) is approximately controllable over the interval \([0, b]\).
	\end{proof} 
	\begin{corollary}
		Assume $A: H\to H$ is a linear bounded operator. The system \eqref{eq:fractional_main} is approximately controllable on $[0, b]$ if 
		\begin{align}\label{t123}
			\overline{span\{A^{m}BU :  m=0,1,2,\dots\}}=H.
		\end{align}
	\end{corollary}
	\begin{proof}
		Assume by contradiction that
		\begin{align*}
			\text{Im} \mathcal{M}=\{x(b)=x(b; 0,u, \{v_{k}\}_{k=1}^{n})\in L^{2}([0,b],U)\times U^{n}\}
		\end{align*}
if it is not dense in $H$, then for some non-zero $\varphi\in H$ $\langle x(b),\varphi\rangle=0$:
		\begin{align*}
			\big\langle x(b),\varphi\big\rangle=&\int_{0}^{b}\Big\langle Bu(s),p(s)\Big\rangle ds
			+\sum_{k=1}^{n}\Big\langle v_{k-1},E^{*}_{k-1}p(t^{+}_{k-1})\Big\rangle,\\
			&\text{for any } (u,\{v_{k}\}_{k=1}^{n})\in L^{2}([0,b],U)\times U^{n},
		\end{align*}
		where 
		
		Given that \(p\) is a solution of equation \eqref{45} in the adjoint equation with \(\varphi \neq 0\), we can easily derive the following:
		
		\[
		\|\mathcal{M}^{*} \phi\|^2 = \Big\langle \left(\Omega^{b}_{t_{n}}+\Psi^{t_{n}}_{0}+ \tilde{\Omega}^{b}_{t_{n}}+\tilde{\Psi}^{t_{n}}_{0}\right)\varphi,\varphi\Big\rangle = 0
		\]
		
		This leads to:
		
		\[
		\Omega_{t_{n}}^{b} = 0 \implies B^{*} (b-s)^{\alpha-1}\mathcal{P}_{\alpha}^{*} (b - s) \varphi = 0, \quad t_{n} \leq s \leq b.
		\]
		
		By successively differentiating this identity, we can show by induction that:
		
		\[
		B^{*} \varphi = B^{*} A^* \varphi = \cdots = B^* (A^*)^m \varphi = 0, \quad m = 0, 1, 2, \ldots
		\]
		
		Therefore, we have:
		
		\[
		0 \neq \varphi \in \bigcap_{m=0}^{\infty} \ker \{B^* (A^*)^m\}
		\]
		
		However, the condition \eqref{t123} is equivalent to:
		
		\[
		\bigcap_{m=0}^{\infty} \ker \{B^* (A^*)^m\} = 0
		\]
		
		(See \cite{6}). This contradiction proves that the system \eqref{eq:fractional_main} is approximately controllable on \([0, b]\).
	\end{proof}	
\section{Application to a Fractional Heat Equation with Impulse}

In this section, we illustrate the applicability of the theoretical results by considering a concrete impulsive fractional heat equation. Let us study the following system:
\begin{align}\label{eq:heat-system}
\begin{cases}
\displaystyle{^{C}D_{t_k}^{\frac{2}{3}} \xi(t,s) = \frac{\partial^{2}}{\partial s^{2}} \xi(t,s) + B u(t,s)}, & t \in [0,1] \setminus \{\frac{1}{2}\}, \; s \in [0,\pi], \; k=0,1,\\[6pt]
\xi(t,0) = \xi(t,\pi) = 0, & t \in [0,1], \\[6pt]
\Delta \xi(\tfrac{1}{2}, s) =  \xi(\tfrac{1}{2}, s) + \sin(s), \\[6pt]
\Delta \xi(1, s) =  \xi(1, s) + \cos(s), \\[6pt]
\xi(0,s) = \xi_0(s).
\end{cases}
\end{align}

Take \(H=U=L^2[0,\pi]\) and let the operator \(A : D(A)\subset H\to H\) be defined by \(Ay=y''\) with domain  
\[
D(A)=\{\xi \in H : \xi,\, \xi^{\prime}\ \text{are absolutely continuous}, \; \xi^{\prime\prime} \in H,\; \xi(0)=\xi(\pi)=0\}.
\]

Then \(A\) can be written as  
\[
Ay=-\sum_{n=1}^{\infty}n^2\langle y, e_n\rangle e_n, \qquad y\in D(A),
\]
where \(e_n(\xi)=\sqrt{\frac{2}{\pi}}\sin (n\xi)\) is an orthonormal basis of \(H\). It is well known that \(A\) is the infinitesimal generator of a compact, analytic, and self‑adjoint semigroup \(\{S(t)\}_{t>0}\) on \(H\) given by  
\[
S(t)y=\sum_{n=1}^{\infty} e^{-n^2 t}\langle y, e_n\rangle e_n, \qquad y\in H.
\]

Define the bounded control operator \(B:U\to H\) by  
\[
Bu(t)=\sum_{n=1}^{\infty}\overline{u}_n(t)\, e_n,
\]
where \(u(t)=\sum_{n=1}^{\infty} \langle u(t), e_n\rangle e_n\) and  
\[
\overline{u}_n(t)=\begin{cases}
0, & 0\leq t<1-\frac{1}{n^2},\\[4pt]
\langle u(t),e_n\rangle, & 1-\frac{1}{n^2}\leq t\leq 1.
\end{cases}
\]

One easily checks that \(\|Bu\|\leq\|u\|\); hence \(B\in\mathcal{L}(U,H)\).

The system \eqref{eq:heat-system} can be rewritten abstractly in \(H\) as  
\begin{align}\label{eq:abs}
\begin{cases}
\displaystyle{^{C}D_{t_k}^{\frac{2}{3}} y(t) = A y(t) + B u(t)}, & t \in [0,1] \setminus \{\frac{1}{2}\}, \; k=0,1,\\[6pt]
\Delta y(\tfrac{1}{2}) =  y(\tfrac{1}{2}) + \sin(\cdot), \\[6pt]
\Delta y(1) =  y(1) + \cos(\cdot), \\[6pt]
y(0) =y_0,
\end{cases}
\end{align}
where \(y(t)=\xi(t,\cdot)\).

For the sake of illustration we assume that the impulse operators are identity operators, i.e. \(D_k=I\) and \(E_k=I\) for \(k=1,2\). Then, using formula (13) with \(n=2\), \(t_1=\frac{1}{2}\), \(t_2=1\), \(b=1\) and \(\alpha=\frac{2}{3}\), we obtain the controllability operator  

\[
\mathcal{M}: L^2([0, 1], U) \times U \to H,
\]
defined by  
\begin{align}\label{eq:Mdef}
\begin{aligned}
\mathcal{M}\big(u(\cdot),v_1,v_2\big) = 
&\int_{\frac{1}{2}}^{1}(1-s)^{-\frac{1}{3}}\mathcal{P}_{\frac{2}{3}}(1-s)B u(s)\,ds  \\
&+ 2\, \mathcal{S}_{\frac{2}{3}}\!\left(\frac{1}{2}\right)\int_{0}^{\frac{1}{2}}\Bigl(\frac{1}{2}-s\Bigr)^{-\frac{1}{3}}\mathcal{P}_{\frac{2}{3}}\Bigl(\frac{1}{2}-s\Bigr)B u(s)\,ds \\
&+\mathcal{S}_{\frac{2}{3}}\!\left(\frac{1}{2}\right)v_1 + v_2 .
\end{aligned}
\end{align}

By Theorem 4, the approximate controllability of \eqref{eq:abs} is equivalent to the condition that \(\mathcal{M}^*h=0\) implies \(h=0\). According to Lemma \ref{lemmaadj}, the adjoint operator is  

\[
\mathcal{M}^*h = \Bigl( B^*p(\cdot),\; p\bigl(\tfrac{1}{2}^+\bigr),\; p(1^+) \Bigr),
\]
where 
\[
B^*p(t)=
\begin{cases}
\displaystyle
B^*(1-t)^{-\frac{1}{3}}\mathcal{P}_{\frac{2}{3}}^*(1-t)h, & \tfrac{1}{2}<t\leq 1,\\[8pt]
\displaystyle
2^{\frac{4}{3}}\,B^*\!\Bigl(\frac{1}{2}-t\Bigr)^{-\frac{1}{3}}
\mathcal{P}_{\frac{2}{3}}^*\!\Bigl(\frac{1}{2}-t\Bigr)
\mathcal{P}_{\frac{2}{3}}^*\!\Bigl(\frac{1}{2}\Bigr)h, & 0<t\le \tfrac{1}{2},
\end{cases}
\]
and the impulsive adjoint terms are
\begin{align}\label{eq:padj}
\begin{cases}
p\bigl(\tfrac{1}{2}^+\bigr)= 2^{\frac{4}{3}}\,\mathcal{P}_{\frac{2}{3}}^*\!\Bigl(\frac{1}{2}\Bigr)h,\\[6pt]
p(1^+)=h .
\end{cases}
\end{align}

Then, for $\frac{1}{2}<t\leq 1$, we have 

\begin{align*}
    &\qquad B^*p(t)=B^*(1-t)^{-\frac{1}{3}}\mathcal{P}_{\frac{2}{3}}^*(1-t)h\\
    &=\frac{2}{3}\sum_{n=1}^{\infty}1_{[1-\frac{1}{n^2},1]}(t) \int_0^\infty \theta\, \Psi_{\frac{2}{3}}(\theta) e^{-n^2 (1-t)^\frac{2}{3}}\langle h, e_n\rangle e_n \, d\theta=0\\
    &\implies \int_0^\infty \theta\, \Psi_{\frac{2}{3}}(\theta) e^{-n^2 (1-t)^\frac{2}{3}} \, d\theta\, \langle h, e_n\rangle=0,\quad t\in \left[1-\frac{1}{n^2},1\right] \\
    &\implies \langle h, e_n\rangle=0\implies h=0.
\end{align*}
For the interval \( 0 \leq t \leq \frac{1}{2} \), we have

\[
\begin{aligned}
B^*p(t) &= 2^{1/3} \sum_{n=1}^{\infty} \mathbf{1}_{[1-\frac{1}{n^2},1]}(t) \, h_n \, k_n\!\left( \frac{1}{2} \right) k_n\!\left( \frac{1}{2} - t \right) e_n, \\
k_n(t) &= \int_0^\infty \frac{2}{3} \theta \, \Psi_{\frac{2}{3}}(\theta) e^{-n^2 t^{2/3}} d\theta > 0.
\end{aligned}
\]

For \(0 \le t \le \frac{1}{2}\), the indicator is nonzero only for \(n=1\), giving  

\[
B^*p(t) = 2^{1/3} \, h_1 \, k_1\!\left( \frac{1}{2} \right) k_1\!\left( \frac{1}{2} - t \right) e_1.
\]

Thus \(B^*p(t)=0 \implies h_1 = 0\). Combined with the result for \(\frac{1}{2} < t \le 1\) (where \(\langle h, e_n \rangle = 0\) for all \(n\)), we obtain \(h = 0\). 
Hence \(\mathcal{M}^* h = 0 \implies h = 0\), proving approximate controllability of \eqref{eq:abs} on \([0,1]\).

\section{Conclusion}

This study establishes a complete characterization of approximate controllability for linear fractional impulsive evolution equations in Hilbert spaces by developing mild solutions through fractional solution operators and deriving necessary and sufficient conditions via the convergence of an associated family of impulsive resolvent operators. The main contributions include a constructive solution approach, an adjoint formulation with Green-type identity, equivalent resolvent-based controllability criteria, and verification through a fractional heat equation example. This work bridges integer-order impulsive theory with fractional dynamics, addressing the interplay of memory effects and instantaneous jumps. Future research may extend these results to stochastic settings, non-local conditions, and inclusion problems, broadening both theoretical foundations and practical applications in anomalous diffusion, viscoelastic materials, and biological systems with memory.

\end{document}